\documentclass[12pt]{amsart}

\usepackage{amssymb,amsfonts}
\usepackage{amsmath,amsthm,setspace}
\usepackage[all]{xy}

\newtheorem{thm}{Theorem}[section]
\newtheorem{lem}[thm]{Lemma}
\newtheorem{prop}[thm]{Proposition}
\newtheorem{cor}{Corollary}

\theoremstyle{definition}
	\newtheorem{defn}[thm]{Definition}
	
	\newtheorem{exmp}[thm]{Example}

\theoremstyle{remark}

\numberwithin{equation}{section}

\newcommand{\TTOs}{\ensuremath{\mathcal{T}_u}}
\newcommand{\TTOsa}{\ensuremath{\mathcal{T}_{u_\alpha}}}

\newcommand{\modsp}[1][u]{\ensuremath{K_{#1}^2}}
\newcommand{\modspa}{\modsp[u_\alpha]}
\newcommand{\modspbdd}{\ensuremath{K_u^\infty}}
\newcommand{\ltwo}{\ensuremath{L^2(\torus)}}
\newcommand{\linfty}{\ensuremath{L^\infty(\torus)}}

\newcommand{\ip}[2]{\ensuremath{\left\langle #1,#2 \right\rangle}}

\newcommand{\CS}{\ensuremath{S_u}}
\newcommand{\CSalpha}[1][\alpha]{\ensuremath{S_u^{#1}}}

\newcommand{\SC}[1]{\ensuremath{\CS\widetilde{#1}}}
\newcommand{\Ku}[1][0]{\ensuremath{K_{#1}^u}}

\newcommand{\CKu}[1][0]{\ensuremath{\widetilde{\Ku[#1]}}}
\newcommand{\TTOeq}{\ensuremath{\stackrel{A}{\equiv}}}
\newcommand{\Balpha}[1][\alpha]{\ensuremath{\mathcal{B}^{#1}}}

\newcommand{\disc}{\ensuremath{\mathbb{D}}}
\newcommand{\cdisc}{\ensuremath{\overline{\disc}}}
\newcommand{\torus}{\ensuremath{\mathbb{T}}}
\newcommand{\plane}{\ensuremath{\mathbb{C}}}
\newcommand{\cplane}{\ensuremath{\plane^*}}

\newcommand{\spKu}[1][0]{\ensuremath{\plane\Ku[#1]}}
\newcommand{\cmt}[1]{\ensuremath{[#1]'}}
\begin{document}

\title[Algebras of TTOs]{Algebras of Truncated Toeplitz Operators}


\author{N.A. Sedlock}
\address{Framingham State University\\Department of Mathematics\\100 State Street, P.O. Box 9101 \\ Framingham, MA 01701-9101}
\email{nsedlock@framingham.edu}


\date{\today}



\begin{abstract}
		We find necessary and sufficient conditions for the product of two truncated Toeplitz operators on a model space to itself
		be a truncated Toeplitz operator, and as a result find a characterization for the maximal algebras of  bounded truncated Toeplitz operators.
\end{abstract}

\maketitle

\section{Introduction}

Let $\plane$ denote the complex plane, $\plane^*$ the Riemann sphere, $\disc$ denote the unit disc, and
let $\torus$ denote the unit circle. $H^2$ is the usual Hardy space, the
subspace of $\ltwo$ of normalized Lebesgue measure $m$ on \torus\ whose harmonic
extensions to $\disc$ are holomorphic (or, whose negative indexed Fourier
coefficients are all zero). $H^2$ will interchangably refer to both the boundary
functions and the functions on $\disc$. Let $P$ denote the projection from
$\ltwo$ to $H^2$, which is given explicitly by the Cauchy integral:
	\[(Pf)(\lambda) =
\int_\torus\frac{f(\zeta)}{1-\lambda\overline{\zeta}}\ dm(\zeta),
\lambda\in\disc.\]The reproducing kernel at
$\lambda\in\disc$ for the Hardy space is the the Szego kernel
$K_\lambda:=(1-\overline{\lambda}z)^{-1}$. 
	$S$ denotes the shift operator $f \mapsto zf$ on $H^2$. Its adjoint (the backward
shift) is the operator\[S^*f = \frac{f-f(0)}{z}.\]

	A Toeplitz operator is the compression of a multiplication operator on
$\ltwo$ to $H^2$. In other words, given $\Phi \in \ltwo$ (called the symbol of
the operator), $T_\Phi = PM_\Phi$ is the operator that sends $f$ to $P(\Phi f)$ for
all $f \in H^2$.  This operator is bounded if and only if $\Phi\in\linfty$, and
the mapping $\Phi \rightarrow T_\Phi$ from $L^\infty$ to the space of bounded
operators on $H^2$ is linear and one-to-one. In the case that $\Phi\in
H^\infty$, the Toeplitz operator $T_\Phi$ is just the multiplication operator $M_\Phi$.
In \cite{brownhalmos}, Brown and Halmos describe the algebraic properties of
Toeplitz operators. Among other things, they found necessary and sufficient
conditions for the product of two Toeplitz operators to itself be a Toeplitz
operator, namely that either the first operator's symbol is antiholomorphic or
the second operator's symbol is holomorphic. In either case, the symbol of the product is the product of the symbols (i.e.  $T_\Phi T_\Psi = T_{\Phi\Psi}$).

	More recently, Sarason~\cite{sarason} found analogues to several of
Brown and Halmos's results for truncated Toeplitz operators on the model spaces
$H^2 \ominus uH^2$, where $u$ is some non-constant inner function. The model
spaces are the backward-shift invariant subspaces of $H^2$ (that they are
backward shift invariant follows easily from the fact that $uH^2$ is clearly
shift invariant). Let \modsp\ denote the space $H^2 \ominus uH^2$ from here
forward. Let $P_u=P - M_uPM_{\overline{u}}$ denote the projection from $L^2$ to
\modsp.

	Given $\Phi\in\ltwo$ we then define the truncated Toeplitz operator (TTO)
$A_\Phi$ to be the operator that sends $f$ to $P_u(\Phi f)$ for all
$f\in\modsp$. $A_\Phi$ is well-defined on the set of bounded functions in $\modsp$, which is dense in $\modsp$ and which we denote $\modspbdd$. We let $\TTOs$ denote the set of truncated Toeplitz operators which extend to be bounded on all of $\modsp$.

 Truncated Toeplitz operators have many of the same properties as
ordinary Toeplitz operators (for example, $A_\Phi^* = A_{\overline{\Phi}}$)
but there are also striking differences. For example, there are
bounded truncated Toeplitz operators with unbounded symbols~\cite{arXiv} (though any
truncated Toeplitz operator with a bounded symbol is itself bounded).  Additionally,
symbols are not unique: the same
operator can be generated from more than one symbol, and we say that $\Psi$ is a
symbol for $A_\Phi$ if $A_\Phi = A_\Psi$. Given two functions $\Psi$ and $\Phi$, we write $\Psi \TTOeq \Phi$ to mean that $A_\Psi = A_\Phi$.

The truncated Toeplitz operators in $\TTOs$ do not form an algebra. There are, however, weakly closed algebras contained in $\TTOs$. The goal of this paper is to describe the maximal algebras contained in $\TTOs$, where by maximal we mean that any weakly closed algebra in $\TTOs$ is contained within one of these maximal algebras. 

In what follows, for functions $f,g$ in $\ltwo$, $\ip{f}{g} = \int_\torus f\overline{g}\ dm$, $\|f\| = \sqrt{\ip{f}{f}}$ and $f\otimes g$ is the rank one operator that maps $h$ to $f\ip{h}{g}$. Further, if $A$ is an operator on a Hilbert space, $\cmt{A}$ denotes the commutant of $A$.

\section{Background}

In this section we lay out basic facts about operators in $\TTOs$ and model spaces.
Let $u$ be a non-trivial inner function. \modsp\ is then a
reproducing kernel Hilbert space with reproducing kernels $\Ku[\lambda]:=P_u
K_\lambda = \frac{1-\overline{u(\lambda)}u}{1-\overline{\lambda}z}$ for
$\lambda\in\disc$. Note that $\Ku[\lambda]$ is bounded for all $\lambda$, and hence in $\modspbdd$.

The function $u$ is said to have an angular derivative in the sense of Caratheodory (ADC) at the point $\zeta\in\torus$ if $u$ has a nontangential limit $u(\zeta)$ of unit modulus at $\zeta$ and $u'$ has a nontangential limit $u'(\zeta)$ at $\zeta$. It is known that $u$ has an ADC at 
$\zeta$ if and only if every function in $\modsp$ has a nontangential limit at $\zeta$~\cite{sarason2}. Thus there exists a reproducing kernel function $\Ku[\zeta]$ such that $\ip{f}{\Ku[\zeta]} = f(\zeta)$. Specifically, $\Ku[\zeta]$ is the limit of $\Ku[\lambda]$ as $\lambda$ approaches $\zeta$ nontangentially in the disc and so $\Ku[\zeta] = \frac{1-\overline{u(\zeta)}u}{1-\overline{\zeta}z}$. If $u$ is a finite Blaschke product, both $u$ and $u'$ are holomorphic in a domain which compactly contains $\disc$ and so these boundary reproducing kernels are defined for every unimodular $\zeta$.

Truncated Toeplitz operators have a symmetry property called $C$-symmetry. This concept is due to Garcia and Putinar~\cite{garcia, garciaputinar1, garciaputinar2}. 
 Given a
\plane-Hilbert space $\mathcal{H}$ and an antilinear isometric involution $C$ on
$\mathcal{H}$, we say that a bounded operator $T$ is a $C$-symmetric operator
(CSO) if $T^* = CTC$. Here by isometric we mean that $\ip{Cf}{Cg} = \ip{g}{f}$.

	In \ltwo, the operator $Cf= u\overline{zf}$ is a conjugation
which bijectively maps $uH^2$ to $\overline{zH^2}$ and \modsp\ to itself. By restricting ourselves to $\modsp$, $C$ can be thought of as a conjugation on $\modsp$. From
here on, $C$ always refers to this operator. We will sometimes write
$\widetilde{f}$ for $Cf$ for sake of readability. The conjugate reproducing kernel is $\CKu[\lambda](z)=
\frac{u(z) - u(\lambda)}{z-\lambda}$ for $z\neq\lambda$ and $\CKu[\lambda](\lambda) = u'(\lambda)$ and has the property that for $f\in\modsp$, $\widetilde{f}(\lambda) = \ip{\CKu[\lambda]}{f}$. 

Consider the operator $C A_\Phi C$, where $\Phi\in\ltwo$ and $A_\Phi\in\TTOs$. If $f,g\in\modsp$, then
\begin{align*}
	\ip{CA_\Phi Cf}{g} &= \ip{Cg}{A_\Phi Cf}\\
			&= \ip{u\overline{zg}}{\Phi u \overline{zf}}\\
			&= \ip{\overline{\Phi} f}{g} \\
			&= \ip{\left(A_\Phi\right)^* f}{g}
\end{align*}
and so we see that operators in $\TTOs$ are $C$-symmetric.

Two CSOs commute if and only if their product is $C$-symmetric.
\begin{prop}\label{prop:CSOcommute}
		Let $A_1$ and $A_2$ be $C$-symmetric. Then $A_1 A_2$ is $C$-symmetric if and only if $A_1$ and $A_2$ commute. 
	\end{prop}
	
	\begin{proof}
		Say $A_1 A_2$ is $C$-symmetric. Then \[A_1A_2 = C A_2^* A_1^* C = C A_2^* C CA_1^* C = A_2 A_1.\]
		
		On the other hand, if $A_1$ and $A_2$ commute, then so do their adjoints, and so
		\[C A_1 A_2 C = A_1^* A_2^* = A_2^* A_1^*.\qedhere\]
	\end{proof}

The operator $\CS = P_u S = A_z$ is critical to what follows. Since $\modsp$ is invariant under $S^*$ we see that $\CS^* = S^*$. Let $f\in\modsp$ such that $f(0) = 0$, i.e.  $f \perp \Ku$. Then $S^*  f = f/z$. On the other hand, $S^*\Ku = (1 - \overline{u(0)}u - 1 + |u(0)|^2)/z = -\overline{u(0)}\CKu$. $\CS$ is $C$-symmetric, and so $\CS$ is characterized by the following equations: $\CS f = zf$ for $f \perp \CKu$, and $\CS\CKu = -u(0)\Ku$.

The symbols of TTOs are a more complex issue than the symbols of Toeplitz operators. Sarason proved the following results in~\cite{sarason} as Theorem 3.1 and Theorem 4.1 respecitively.
	\begin{prop}
	If $\Phi\in\ltwo$ then  $A_\Phi = 0$ if and only if $\Phi \in uH^2 + {\overline{uH^2}}$
\label{prop:TTOzerocond}.
	\end{prop}
	\begin{prop}
		$A$ is in $\TTOs$ iff $A - \CS A\CS^* = \Phi\otimes \Ku  + \Ku
\otimes\Psi$ for some $\Phi, \Psi \in \modsp$, in which case $A = A_{\Phi +
\overline{\Psi}}$. \label{prop:TTOchar}
	\end{prop}

	Thus we have a way of finding a symbol for a TTO, but TTOs do not have unique symbols. We say that $\Phi \TTOeq \Psi$ if and only if $A_\Phi = A_\Psi$.
	
	The following is a necessary and sufficient condition for a TTO with symbol in $\modsp + \overline{\modsp}$ to equal zero.
	\begin{prop}
	 	Let $\varphi_1, \varphi_2\in\modsp$. Then $A_{\varphi_1 +
\overline{\varphi_2}} = 0$ if and only if
		$\varphi_1 = c\Ku$ and $\varphi_2 = -\overline{c}\Ku$ for some
$c\in\plane$. \label{prop:TTOhardzerocond}
	\end{prop}
	
	\begin{proof}
	 	Let $\varphi_1 = c\Ku$ and $\varphi_2 = -\overline{c}\Ku$. Then
\[A_{\varphi_1 + \overline{\varphi_2}} = A_{c\Ku - c\overline{\Ku}} =
A_{cu(z)\overline{u(0)}-c\overline{u(z)}u(0)}\] so $A_{\varphi_1 +
\overline{\varphi_2}} = 0$.

		Now suppose $A_{\varphi_1 + \overline{\varphi_2}} = 0$. Then $A - \CS A \CS^* = 0 = \varphi_1 \otimes\Ku + \Ku\otimes\varphi_2$, so $\varphi_1 = c \Ku$ for some $c\in\plane$. Hence $c \Ku\otimes\Ku + \Ku\otimes\varphi_2 = 0$ and so $\varphi_2 = - c\Ku$ as required. 
	\end{proof}
	 
		Since $I = A_{\Ku}$ we can compute the identities \begin{equation}I - \CS\CS^* = \Ku\otimes\Ku~\label{eqn:I-SS*}\end{equation} and \begin{equation}I - \CS^*\CS = \CKu\otimes\CKu~\label{eqn:I-S*S}\end{equation} from which it follows that \begin{equation}\label{eqn:SCSC} \SC{\SC{\varphi}} = \CS\CS^*\varphi = \varphi - \varphi(0)\Ku\end{equation} for all $\varphi\in\modsp$.

	The following identities are Lemma 2.2 of ~\cite{sarason}.
	\begin{prop}\label{prop:CSKu}\mbox{ }
	 \begin{enumerate}
		\item If $\lambda\in\disc$,  \[\CS^*\Ku[\lambda] = \overline{\lambda}\Ku[\lambda] - 
												\overline{u(\lambda)}\CKu\] and \[\ \CS\CKu[\lambda] = \lambda\CKu[\lambda] - u(\lambda)\Ku . \]
		\item If $\lambda\in\disc$ is nonzero, 
				\[
					\CS\Ku[\lambda] = \frac{1}{\overline{\lambda}}\left(\Ku[\lambda] - \Ku\right)\] and \[\ \CS^*\CKu[\lambda] = \frac{1}{\lambda}\left(\CKu[\lambda]-\CKu\right)
				.\]

		\item These equalities all hold for $\lambda\in\torus$ such that $u$ has an ADC at $\lambda$.
	 \end{enumerate}
	\end{prop}
	
\section{Generalized Shifts}

We now define the generalized compressed shift operator. Our definition follows Sarason's definition in Section $14$ of~\cite{sarason}.
\begin{defn}\label{defn:CS}
	Let $\alpha\in\cdisc$. Then $\CSalpha = \CS + \frac{\alpha}{1-\alpha\overline{u(0)}}\Ku\otimes\CKu$. 
\end{defn}

Again, we can think about the generalized shift as follows. If $f\in\modsp$ and $f\perp \CKu$, then $\CSalpha f = zf$. On the other hand, \begin{align*}\CSalpha\CKu &= \CS\CKu + \frac{\alpha\ip{\CKu}{\CKu}}{1-\alpha\overline{u(0)}}\Ku\\ &= -u(0)\Ku + \frac{\alpha(1-|u(0)|^2)}{1-\alpha\overline{u(0)}}\Ku \\&= \frac{\alpha - u(0)}{1-\alpha\overline{u(0)}}\Ku.\end{align*}

The corollary to Theorem 10.1 in~\cite{sarason} states that if a bounded operator $A$ on $\modsp$ is in $\cmt{\CSalpha}$ then $A$ is in $\TTOs$. The following proof gives us the symbol of any operator in $\cmt{\CSalpha}$.
\begin{prop}\label{prop:CSalphaCommutant}
	Let $\alpha\in\cdisc$. If $A$ is a bounded operator that commutes with $\CSalpha$ then $A$ is in $\TTOs$ and has a symbol $\varphi+\alpha\overline{\SC{\varphi}}$ where $\varphi = A\Ku(1-\alpha\overline{u(0)})^{-1}$. 
\end{prop}

\begin{proof}
  First note that \begin{equation}\label{eqn:ACSalpha}A\CSalpha = A\CS + \frac{\alpha}{1 -
\alpha\overline{u(0)}}\left(A\Ku\right)\otimes\CKu\end{equation} and \begin{align}\CSalpha A &= \CS A +
\frac{\alpha}{1 - \alpha\overline{u(0)}}\Ku\otimes\left(A^*\CKu\right)\nonumber\\ &= \CS A +
\frac{\alpha}{1 -
\alpha\overline{u(0)}}\Ku\otimes\left(\widetilde{A\Ku}\right)\label{eqn:CSalphaA}.\end{align} If $A$ and $\CSalpha$ commute then we can use Equations (\ref{eqn:ACSalpha}) and (\ref{eqn:CSalphaA}) to see that 
\begin{align*}
	\CS A &= \CSalpha A - \frac{\alpha}{1-\alpha\overline{u(0)}}\Ku\otimes\left(\widetilde{A\Ku}\right)\\
		&= A \CSalpha - \frac{\alpha}{1-\alpha\overline{u(0)}}\Ku\otimes\left(\widetilde{A\Ku}\right)\\
		&=  A\CS + \frac{\alpha}{1 -
\alpha\overline{u(0)}}\left(A\Ku\right)\otimes\CKu- \frac{\alpha}{1-\alpha\overline{u(0)}}\Ku\otimes\left(\widetilde{A\Ku}\right).
\end{align*}

It follows that
  \begin{align*}
    A - \CS A \CS^* ={}& A - A \CS \CS^* - \frac{\alpha}{1 -
\alpha\overline{u(0)}}A\Ku\otimes\SC{\Ku} \\&+ \frac{\alpha}{1 -
\alpha\overline{u(0)}}\Ku\otimes\SC{A\Ku}\\
    ={}& A\Ku\otimes\Ku + \frac{\overline{u(0)}\alpha}{1 -
\alpha\overline{u(0)}}A\Ku\otimes\Ku \\&+ \frac{\alpha}{1 -
\alpha\overline{u(0)}}\Ku\otimes\SC{A\Ku}\\
    ={}& \frac{A\Ku}{1-\alpha\overline{u(0)}}\otimes\Ku \\&+
\Ku\otimes\overline{\alpha}\CS
C\left({\frac{A\Ku}{1-\alpha\overline{u(0)}}}\right).
  \end{align*}
The conclusion then follows from Proposition~\ref{prop:TTOchar}.
\end{proof}

\begin{cor}
Let $A$ be a bounded opeator that commutes with ${\CSalpha}^*$, for $\alpha\in\cdisc$. Then $A$ is in $\TTOs$ and has a symbol of the form $\overline{\alpha}\psi + \overline{\SC{\psi}} + c$ for $\psi\in\modsp$ and $c\in\plane$.
\end{cor}
\begin{proof}
$A^*$ commutes with $\CSalpha$ and therefore has symbol $\varphi + \alpha\overline{\SC{\varphi}}$ where $\varphi = A^*\Ku (1-\alpha\overline{u(0)})$ by the previous proposition. Therefore $A$ has symbol $\overline{\alpha}\SC{\varphi} + \overline{\varphi}$. Define $\psi = \SC{\varphi}$. Then by Equation~\ref{eqn:SCSC} $\SC{\psi} = \SC{\SC{\varphi}} = \varphi - \varphi(0)\Ku$ and $\overline{\alpha}\psi + \overline{\SC{\psi}}+ \overline{\varphi(0)}$ is a symbol for $A$. 
\end{proof}

Suppose $A_\Phi, A_\Psi$ are in $\TTOs$ and both commute with $\CSalpha$ for some $\alpha\in\cdisc$. Then their product $A_\Phi A_\Psi$ also commutes with $\CSalpha$, and is therefore also in $\TTOs$. So we know of two cases when the product of two operators in $\TTOs$ is itself in $\TTOs$ --- when both operators commute with some $\CSalpha$ or ${\CSalpha}^*$, or when one of the operators is $A_c = cI$ for some $c\in\plane$. We will show in Section~\ref{section:algebras} that these are the only cases where the product of two operators in $\TTOs$ is itself in $\TTOs$.

\section{TTOs of type $\alpha$}

If $A_\Phi$ is in $\TTOs$ and commutes with $\CSalpha$, then $A_{\Phi + c}$ also commutes with $\CSalpha$ for all $c\in\plane$. If $\alpha\in\cdisc\setminus\{0\}$, then $\overline{\alpha}^{-1}\in\plane\setminus\disc$, and by the corollary to Proposition~\ref{prop:CSalphaCommutant} any operator in $\TTOs$ which commutes with ${\CSalpha}^*$ has a symbol of the form $\psi + \overline{\alpha}^{-1}\overline{\SC{\psi}} + c$ with $\psi\in\modsp$ and $c\in\plane$. We therefore make the following definition.

\begin{defn}
	 An operator $A\in\TTOs$ is said to be a TTO of type $\alpha$ for $\alpha\in\plane$ if $A$ has a symbol of the form $\varphi + \alpha\overline{\SC{\varphi}} + c$, where $\varphi\in\modsp$ and $c\in\plane$. Note that an operator in $\TTOs$ is of type $0$ if and only if it has a holomorphic symbol. We say an operator in $\TTOs$ is of type $\infty$ if it has an antiholomorphic symbol.
\end{defn}

\begin{prop}\label{prop:inBalpha}Let $A:=A_{\varphi_1 +
\overline{\varphi_2}}$ be in $\TTOs$, where $\varphi_i\in\modsp$.\begin{enumerate}
		\item If $\alpha\in\plane$, then $A$ is of type $\alpha$ if and only if $\overline{\alpha}\SC{\varphi_1}-\varphi_2 \in \spKu$.
		\item $A$ is of type $\infty$ if and only if $\varphi_1\in\spKu$ if and only if $\SC{\varphi_1}\in\spKu$.
	\end{enumerate}
	\end{prop}
	\begin{proof}\mbox{}
	\begin{enumerate}
	\item Let $A_{\varphi_1 + \overline{\varphi_2}}$ be of type $\alpha$. Then by Proposition~\ref{prop:CSalphaCommutant} and its corollary
there is some $\varphi\in\modsp$ and $c \in\plane$ such that $A_{\varphi_1 +
\overline{\varphi_2}} = A_{\varphi + c\Ku + \alpha\overline{\SC{\varphi}}}$, or, equivalently \[A_{\varphi_1 - \varphi -c\Ku + \overline{\varphi_2} -
\alpha\overline{\SC{\varphi}}}=0\] By Proposition~\ref{prop:TTOhardzerocond}
we have that $\varphi_1 - \varphi \in\spKu$ and that $\varphi_2 -
\overline{\alpha}\SC{\varphi}\in\spKu$. So then by Proposition~\ref{prop:CSKu} we have that $\SC{\varphi_1} -
\SC{\varphi}\in\spKu$ and so $
\overline{\alpha}\SC{\varphi_1}-\varphi_2 = \overline{\alpha}\SC{\varphi_1} -
\overline{\alpha}\SC{\varphi} - \varphi_2 + \overline{\alpha}\SC{\varphi} \in \spKu$.

		Now suppose that $\overline{\alpha}\SC{\varphi_1}-\varphi_2 \in
\spKu$. Then $\varphi_2 = \overline{\alpha}\SC{\varphi_1} + c\Ku$ for some
$c\in\plane$ and thus $A_{\varphi_1 + \overline{\varphi_2}} = A_{\varphi_1 +
\alpha\overline{\SC{\varphi_1}}+\overline{c\Ku}}$ is of type $\alpha$.
	\item $A$ is of type $\infty$ if and only if $\varphi_1 + \overline{\varphi_2} \TTOeq \overline{\psi}$ for some $\psi\in\modsp$, which is true if and only if $\varphi_1 = P_u (\overline{\psi -\varphi_2}) \TTOeq  \overline{\psi(0) - \varphi_2(0)}$ which is true if and only if $\varphi_1 \in \spKu$.

If $\varphi_1 = c\Ku$ then $\SC{\varphi_1} = -\overline{c}u(0)\Ku$ by Proposition~\ref{prop:CSKu}. On the other hand, if $\SC{\varphi_1} = c\Ku$ then \begin{align*}
\varphi_1 &= (\CS\CS^* - \Ku\otimes\Ku)\varphi_1 \\&=\SC{\SC{\varphi_1}}-\varphi_1(0)\Ku\\&=\SC{c\Ku}-\varphi_1(0)\Ku\\&=-\overline{c}u(0)\Ku - \varphi_1(0)\Ku\\&\in\spKu \qedhere
\end{align*}
	\end{enumerate}
	\end{proof}
	
\begin{prop}\label{prop:phi0}
Any TTO of type $\alpha\in\plane$ has a symbol of the form $\varphi_0 + \alpha\overline{\SC{\varphi_0}} + c\Ku$ where $\varphi_0(0) = 0$ and $c\in\plane$, and any TTO of antiholomorphic type has a symbol of the form $\overline{\varphi_0} + c\Ku$ where $\varphi_0(0)=0$.
\end{prop}
\begin{proof} To prove the first statement, let $A$ be of type $\alpha\in\plane$ and let $\varphi + \alpha\overline{\SC{\varphi}} + c\Ku$ be a symbol of $A$, where $\varphi\in\modsp$ and $c\in\plane$. Define $\varphi_0 = \varphi - \frac{\ip{\varphi}{\Ku}}{\ip{\Ku}{\Ku}}\Ku$. Then $\varphi_0 \perp \Ku$, or in other words, $\varphi_0(0) = 0$. Then since by Proposition~\ref{prop:CSKu} $\SC{\Ku} = -u(0)\Ku$ we have that 
	\begin{align*}\varphi + \alpha\overline{\SC{\varphi}} + c\Ku \TTOeq{}& \varphi_0 + \frac{\ip{\varphi}{\Ku}}{\ip{\Ku}{\Ku}}\Ku + \alpha\overline{\SC{\varphi_0}} \\&+ \alpha\frac{\overline{\ip{\varphi}{\Ku}}}{\ip{\Ku}{\Ku}}\overline{\Ku} + c\Ku\\\TTOeq{}& \varphi_0 + \alpha\overline{\SC{\varphi_0}} + c_1\Ku\end{align*} where $c_1\in\plane$. 

To prove the second statement, consider $A = A_{\overline{\varphi}}$ and let $\varphi_0 = \varphi - \frac{\ip{\varphi}{\Ku}}{\ip{\Ku}{\Ku}}\Ku$. Then $\overline{\varphi} \TTOeq \overline{\varphi_0} +  \frac{\ip{\Ku}{\varphi}}{\ip{\Ku}{\Ku}}\Ku$.
\end{proof}
	
	Let $\alpha\in\plane\setminus\{0\}$. Then if $A=A_{\varphi_1 + \overline{\varphi_2}}$ is of type $\alpha$, its adjoint is $A^*=A_{\psi_1 + \overline{\psi_2}}$ where $\psi_1 = \varphi_2$ and $\psi_2 = \varphi_1$. By Proposition~\ref{prop:inBalpha} it follows that \[\overline{\alpha}\CS{\psi_2} - \psi_1 \in \spKu.\] It follows by Proposition~\ref{prop:CSKu} that
\begin{align*}  \CS C(\overline{\alpha}\SC{\psi_2}- \psi_1) &=  \alpha \CS \CS^* \psi_2- \SC{\psi_1}\\
&= \alpha \psi_2 - \SC{\psi_1} + \alpha\ip{\psi_2}{\Ku}\Ku\\&\in\spKu.\end{align*} The second equation follows from Equation~\ref{eqn:SCSC}. Hence we have that $\alpha^{-1}\SC{\psi_1} - \psi_2\in\spKu$ and so it follows that $A^*$ is of type $\overline{\alpha^{-1}}$. In the case that $A$ is of type 0, $A$ has a holomorphic symbol, and so its adjoint $A^*$ has an antiholomorphic symbol, and is therefore of type $\infty$. Thus we can state the following duality relationship.

\begin{prop}
	An operator in $\TTOs$ is of type $\alpha\in\cplane$ if and only if its adjoint is of type $\overline{\alpha^{-1}}$ using the convention that $0^{-1} = \infty$ and $\infty^{-1} = 0$.
\end{prop}

The operator $cI = A_{c\Ku} = A_{c\overline{\Ku}}$ is, by the above definition, of type $\alpha$ for every $\alpha\in\cplane$. This is the only way that an operator in $\TTOs$ can be of more than one type. Specifically, this means that any $A\in\TTOs$ is either of no type, one type, or every type.

	\begin{prop}
		Let $A\in\TTOs$ be of type $\alpha$ and of type $\beta$, where $\alpha\neq\beta$. Then $A = cI$ for some $c\in\plane$.
	\end{prop}
	\begin{proof}
		If $\alpha =0$ and $\beta = \infty$, then there are $\varphi, \psi\in\modsp$ such that $A = A_\varphi = A_{\overline{\psi}}$ and so $A_\varphi - \CS A_\varphi \CS^* = \varphi\otimes\Ku$ and $A_{\overline{\psi}} - \CS A_{\overline{\psi}} \CS^* = \Ku \otimes\psi$ by Proposition~\ref{prop:TTOchar}. Thus $\varphi\otimes\Ku = \Ku\otimes\psi$ and $\varphi = c\Ku$ for some $c\in\plane$, and so $A = cI$.
		
		Now suppose that at least one of $\alpha$ and $\beta$ is in $\plane\setminus\{0\}$. By looking at $A^*$ if needed we can assume without loss of generality that neither $\alpha$ or $\beta$ is $\infty$. By Proposition~\ref{prop:phi0} there are $\varphi,\psi\in\modsp$ and $c,d\in\plane$ such that $\varphi(0) = \psi(0) = 0$  and both $\varphi + \alpha\overline{\SC{\varphi}} + c$ and $\psi + \beta\overline{\SC{\psi}} + d$ are symbols for $A$. It follows that \begin{align*} A - \CS A \CS^*&= \varphi\otimes\Ku + c\Ku\otimes\Ku + \alpha\Ku\otimes\SC{\varphi} \\&= \psi\otimes\Ku + d\Ku\otimes\Ku + \beta\Ku\otimes\SC{\psi}.\end{align*} By rearranging terms we see that $\varphi - \psi \in \spKu$. Since $\varphi, \psi \perp \Ku$ it follows that $\varphi = \psi$ and \[ (c-d) \Ku \otimes\Ku = (\beta - \alpha)\Ku\otimes\SC{\varphi}.\] Therefore $\SC{\varphi} = \frac{c-d}{\beta-\alpha}\Ku$ but since \[\ip{\SC{\varphi}}{\Ku} = \ip{\CKu}{\CS^* \varphi} = \ip{\SC{\Ku}}{\varphi} = \ip{-u(0)\Ku}{\varphi} = 0\] we get that $c = d$ and $\SC{\varphi} = 0$. 
		
		  Finally we calculate $\varphi = (I - \Ku\otimes\Ku)\varphi = \SC{\SC{\varphi}} = 0$ and get that $A = A_c = c_I$.  
	\end{proof}
	
 For the rest of this section fix $\alpha\in\cdisc$. By Proposition~\ref{prop:CSalphaCommutant} if an operator $A\in\TTOs$ is in $\cmt{\CSalpha}$ then it is of type $\alpha$. We spend the remainder of this section proving that every TTO of type $\alpha$ is in $\cmt{\CSalpha}$. Specifically, we will show that the product of two TTOs of type $\alpha$ is itself in $\TTOs$. Therefore any two TTOs of type $\alpha$ commute and so any TTO of type $\alpha$ commutes with $\CSalpha$. Therefore for $\alpha\in\cdisc$, $\cmt{\CSalpha}$ is precisely the TTOs of type $\alpha$, and therefore $\cmt{{\CSalpha}^*}$ is precisely the TTOs of type $\overline{\alpha}^{-1}$ with the convention that $\frac{1}{0} = \infty$.

First a lemma that will prove useful here and later.
	\begin{lem}\label{lem:TTOhardprodcondition}Let $\Phi = \varphi_1+\overline{\varphi_2}$ and $\Psi =
\psi_1+\overline{\psi_2}$ where $\varphi_i, \psi_i \in \modsp$ such that $A_\Phi,A_\Psi\in\TTOs$. Then $A_\Phi
A_\Psi$ is in $\TTOs$ if and only if
		\[\varphi_1\otimes \psi_2 -
(\SC{\varphi_2})\otimes(\SC{\psi_1}) = \Phi_0\otimes \Ku  + \Ku \otimes\Psi_0\]
		 for some $\Phi_0, \Psi_0 \in \modsp$.\end{lem}
	\begin{proof}
		In what follows, $\Phi_0$ and $\Psi_0$ represent functions in
\modsp\ that can be different
		from use to use.
		By Proposition~\ref{prop:TTOchar}, $A_\Phi A_\Psi\in\TTOs$ if and only
if $A_\Phi A_\Psi - \CS 
		A_\Phi A_\Psi \CS^* = \Phi_0\otimes\Ku + 
		\Ku\otimes\Psi_0$. It suffices to show that $A_\Phi A_\Psi -
\CS 
		A_\Phi A_\Psi \CS^* = \varphi_1\otimes \psi_2 - 
		(\SC{\varphi_2})\otimes(\SC{\psi_1}) + 
		\Phi_0\otimes \Ku  + \Ku \otimes\Psi_0$.
		Recall Equation~\ref{eqn:I-S*S}, which states that $I = \CS^*\CS + 
		\CKu\otimes\CKu$. Therefore
		\begin{align}
			 \CS A_\Phi A_\Psi \CS^* ={}& \CS A_\Phi (\CS^* \CS + \CKu\otimes\CKu) A_\Psi \CS^*\nonumber\\
						={}& \CS A_\Phi \CS^* \CS A_\Psi \CS^* + \left(\CS A_\Phi \CKu\right)\otimes
			\left(\CS A_{\overline{\Psi}} \CKu\right).~\label{hugemess1}
		\end{align}
	Since $A_\Phi \CKu = P_u\left[\left(\varphi_1 + \overline{\varphi_2}\right)\left(\overline{z}\left(u - u(0)\right)\right)\right]$ we have
	\begin{align*}
		\CS A_\Phi \CKu
		&= 	\CS \left(\widetilde{\varphi_2} + \varphi_1(0)\CKu
			-u(0)\CS^*\varphi_1\right)\\
		&=  \SC{\varphi_2} - u(0) \varphi_1(0)\Ku
			-u(0)\CS\CS^*\varphi_1\\
		&=	\SC{\varphi_2} - u(0) \varphi_1(0)\Ku
			-u(0)\varphi_1 +
u(0)\left(\Ku\otimes\Ku\right)\varphi_1\\
		&=	\SC{\varphi_2} - u(0) \varphi_1(0)\Ku
			-u(0)\varphi_1 + u(0)\varphi_1(0)\Ku\\
		&=  \SC{\varphi_2}-u(0)\varphi_1
	\end{align*}
	
	so the second term of (\ref{hugemess1}) is
	\begin{align*}
		\left(\CS A_\Phi \CKu\right)\otimes
		\left(\CS A_{\overline{\Psi}} \CKu\right)
	={}&\left(\SC{\varphi_2}-u(0)\varphi_1\right)\otimes
		\left(\SC{\psi_1}-u(0)\psi_2\right)\\
	={}&\SC{\varphi_2}\otimes\SC{\psi_1} -
		u(0)\left[\varphi_1\otimes\SC{\psi_1}\right] \\
	&-	\overline{u(0)}\left[\SC{\varphi_2}\otimes\psi_2\right] +
		|u(0)|^2\left[\varphi_1\otimes\psi_2\right].
	\end{align*}
	By Proposition~\ref{prop:TTOchar} we have
	that $\CS A_\Phi \CS^* = A_\Phi - \varphi_1\otimes\Ku -
\Ku\otimes\varphi_2$, and so the first term of (\ref{hugemess1}) is

		\begin{align*}
			\CS A_\Phi \CS^* \CS A_\Psi \CS^* ={}& \left(A_\Phi - \varphi_1\otimes\Ku -
\Ku\otimes\varphi_2\right)
			\left(A_\Psi - \psi_1\otimes\Ku -
\Ku\otimes\psi_2\right)\\
		={}& A_\Phi A_\Psi - \Phi_0\otimes\Ku -
\left(A_\Phi\Ku\right)\otimes\psi_2\\
	&- \varphi_1\otimes\left(A_{\overline{\Psi}}\Ku\right) + \left(1 - |u(0)|^2\right)
\varphi_1\otimes\psi_2
	- \Ku\otimes\Psi_0 \\
	={}& A_\Phi A_\Psi + \Phi_0\otimes\Ku
- \Ku\otimes\Psi_0-\left(1 + |u(0)|^2\right) \varphi_1\otimes\psi_2\\
	&+ \overline{u(0)}\left(\SC{\varphi_2}\otimes\psi_2\right)
		+u(0)\left(\varphi_1\otimes\SC{\psi_1}\right).
		\end{align*}

	By combining the expanded terms together, we get 
	\[
			\CS A_\Phi A_\Psi \CS^* 
			= \SC{\varphi_2}\otimes\SC{\psi_1}\ - \varphi_1\otimes\psi_2
			+\Phi_0\otimes\Ku + \Ku\otimes\Psi_0 + A_\Phi A_\Psi
		\]
		and the result follows.
	\end{proof}

\begin{thm}~\label{thm:CSalphaCommutant}
	Let $\alpha\in\cdisc$, and let $A$ be a bounded operator on $\modsp$. Then $A$ is a TTO of type $\alpha$ if and only if $A$ is in $\cmt{\CSalpha}$.
\end{thm}
\begin{proof}
Proposition~\ref{prop:CSalphaCommutant} proves that everything in the $\cmt{\CSalpha}$ is of type $\alpha$, so assume $A$ is of type $\alpha$. We will prove that $A\CSalpha$ is in $\TTOs$, and hence $C$-symmetic, and so $A\CSalpha = \CSalpha A$ by Proposition~\ref{prop:CSOcommute}.

$\CSalpha$ commutes with itself, and therefore is of type $\alpha$. By Definition~\ref{defn:CS} \[\CSalpha\Ku = \CS\Ku + \frac{\alpha\overline{u'(0)}}{1-\alpha\overline{u(0)}}\Ku.\] So by Proposition~\ref{prop:CSalphaCommutant} \[(1-\alpha\overline{u(0)})^{-1}(\CS\Ku + \alpha\overline{\SC{\CS\Ku}} + \frac{\alpha\overline{u'(0)}}{1-\alpha\overline{u(0)}}(\Ku + \alpha\overline{\SC{\Ku}}))\] is a symbol for \CSalpha. By Proposition~\ref{prop:CSKu} \[\Ku + \alpha\overline{\SC{\Ku}} \TTOeq (1-\alpha\overline{u(0)})\] and so it follows that \begin{equation}(1-\alpha\overline{u(0)})^{-1}(\CS\Ku + \alpha\overline{\SC{\CS\Ku}} + \alpha\overline{u'(0)}\Ku)\label{eqn:CSalphaSymbol}\end{equation} is also a symbol for $\CSalpha$. 

Suppose $A$ is of type $\alpha$. Then we may without loss of generality assume that $\varphi + \alpha\overline{\SC{\varphi}}$ is a symbol for $A$ where $\varphi$ is in $\modsp$. Applying Lemma~\ref{lem:TTOhardprodcondition} we see that $A \CSalpha$ is in \TTOs\ if and only if there exist $\Phi, \Psi\in\modsp$ such that \[\varphi\otimes \left(\overline{\alpha}\SC{\CS\Ku}\right) - \left(\SC{\overline{\alpha}\SC{\varphi}}\right)\otimes \SC{\CS\Ku} = \Phi \otimes\Ku + \Ku\otimes\Psi\] Factoring $\alpha$ out of the left-hand side, we get
\begin{align*}
	\varphi\otimes \left(\SC{\CS\Ku}\right) - \left(\SC{\SC{\varphi}}\right)\otimes \SC{\CS\Ku} &= \left(\left(I - \CS\CS^*\right)\varphi\right) \otimes \SC{\CS\Ku}
	\\&= \varphi(0)\Ku\otimes\SC{\CS\Ku}
\end{align*} The conclusion follows.
\end{proof}

\section{Algebras of TTOs}\label{section:algebras}
The results of the previous section show that the TTOs of type $\alpha$ form a weakly closed commutative algebra for any $\alpha\in\cplane$, which we denote $\Balpha$. In this section we will show that these algebras are maximal --- any algebra in $\TTOs$ is a subalgebra of at least one $\Balpha$. 

We begin by showing that if $A_\Phi$ is of type $\alpha$, $A_\Psi\in\TTOs$, and their product is in $\TTOs$, then either $A_\Phi$ is a multiple of $I$, or $A_\Psi$ is of type $\alpha$ as well.

\begin{lem}
	Let $A_\Phi,A_\Phi\in\TTOs$ such that $A_\Phi A_\Psi\in\TTOs$ and let $\alpha\in\plane^*$. If one of the
operators in the 
	product is of type $\alpha$, then either it is a constant multiple of the
identity operator, or the other
	is of type $\alpha$ as well.\label{lem:TTOBalphaprod}
\end{lem}
\begin{proof}
	Since $A_\Phi A_\Psi$ is in $\TTOs$, it is a CSO, and so $A_\Phi A_\Psi = A_\Psi A_\Phi$ by Proposition~\ref{prop:CSOcommute}. Thus we
assume without loss of generality that
	$A_\Phi$ is of type $\alpha$. Additionally $A_\Phi A_\Psi$ is in $\TTOs$ if and only if its adjoint $C A_\Phi A_\Psi C = A_{\overline{\Phi}} A_{\overline{\Psi}}$ is as well, where $A_{\overline{\Phi}}$ is of type $\overline{\alpha}^{-1}$, so we assume without loss of generality that $A_\Phi$ is of type $\alpha\in\cdisc$. So $\Phi\TTOeq\varphi_0 + \alpha\overline{\SC{\varphi_0}} +
c\Ku$ and $\Psi
	\TTOeq \psi_1 + \overline{\psi_2}$ for some $\varphi_0, \psi_1, \psi_2
\in \modsp$,
	where by Proposition~\ref{prop:phi0}  we may assume that $\varphi_0(0)=0$, $c\in\plane$. By Lemma~\ref{lem:TTOhardprodcondition},
there exists $\Phi_0,
	\Psi_0\in\modsp$ such that
	\begin{align*}
		\Phi_0\otimes\Ku + \Ku\otimes\Psi_0
		&=\left(\varphi_0 + c\Ku\right)\otimes\psi_2 - 
	\left(\SC{\left(\overline{\alpha}\SC{\varphi_0}\right)}\right)
	\otimes\left(\SC{\psi_1}\right)\\
		&=\varphi_0\otimes\psi_2 + c\Ku\otimes\psi_2 - 
	\varphi_0	\otimes\left(\overline{\alpha}\SC{\psi_1}\right)\\
&=\varphi_0\otimes\left(\psi_2-\overline{\alpha}\SC{\psi_1}\right) + 	
		c\Ku\otimes\psi_2
	\end{align*}
	So
$\varphi_0\otimes\left(\psi_2-\overline{\alpha}\SC{\psi_1}\right)
	=\Phi_0\otimes\Ku + \Ku\otimes\Psi_1$ for some $\Psi_1\in\modsp$. So either $\Phi_0$ and \Ku\ are
linearly dependent or
	$\Psi_1$ and \Ku\ are. If $\Phi_0$ and \Ku\ are linearly
dependent, then $\Phi_0 =
	c_1\Ku$ which means $\varphi_0 = c_2\Ku$, but this and $\varphi_0(0)= 0$
then imply that
	$c_2 = 0$, and so $\varphi_0 = 0$ and $A_\Phi = cI$.
	Otherwise, $\Psi_1 = c_3\Ku$ and so $ \psi_2-\overline{\alpha}\SC{\psi_1} =
c_4\Ku$,
	which means $A_\Psi$ is of type $\alpha$ by Proposition~\ref{prop:inBalpha}.
	\end{proof}	

We now prove the main theorem of this section.

\begin{thm}\label{thm:result1}
		Let $\Phi, \Psi\in \ltwo$ such that $A_\Phi,A_\Psi\in\TTOs$. Then $A_\Phi 
		A_\Psi\in\TTOs$ if and only if one of two (not mutually
exclusive) cases holds:

		Trivial case: Either $A_\Phi$ or $A_\Psi$ is equal to $cI$ for
some $c\in\plane$.

		Non-trivial case: $A_\Phi$ and $A_\Psi$ are both of type $\alpha$ for
some $\alpha\in\plane^*$, in which case their product is of type $\alpha$ as well.

		\end{thm}
	\begin{proof}The sufficiency of either case follows from earlier discussion, so we prove their necessity. In what follows we will use the fact that if $\Phi$ and $\Psi$ are functions such that $A_\Phi A_\Psi\in\TTOs$, then for any complex constants $c_1, c_2$ 
$A_{\Phi+c_1} A_{\Psi+c_2}\in\TTOs$.

		Suppose $A_\Phi A_\Psi\in\TTOs$. By Lemma~\ref{lem:TTOBalphaprod} it suffices to show that one of $A_\Phi$ and $A_\Psi$ is of type $\alpha$ for some $\alpha\in\cplane$.
		
		There exists $\varphi_i, \psi_i\in\modsp$ such that
		we may assume without loss of generality 
		that $\Phi = \varphi_1 + \overline{\varphi_2}$ and that $\Psi =
\psi_1 + 
		\overline{\psi_2}$. Then it follows by
Lemma~\ref{lem:TTOhardprodcondition} that 
		\[
			\varphi_1\otimes \psi_2-
		(\SC{\varphi_2})\otimes(\SC{\psi_1})
 		= \Phi_0\otimes \Ku  + \Ku \otimes\Psi_0
		\] holds for some $\Phi_0, \Psi_0$ in \modsp. If at least one of $\Phi_0$ and $\Psi_0$ is non-zero, but one of them is in \spKu, then the right-hand side of this equation is a rank one operator $f\otimes g$. Thus we consider the following three cases.
	\begin{enumerate}
		\item $\varphi_1\otimes \psi_2
			-(\SC{\varphi_2})\otimes(\SC{\psi_1}) = 0$
		\item $\varphi_1\otimes \psi_2
			-(\SC{\varphi_2})\otimes(\SC{\psi_1}) =f \otimes g;\ f,g\in\modsp$		 
		\item 	$\varphi_1\otimes \psi_2-
		(\SC{\varphi_2})\otimes(\SC{\psi_1})
 		= \Phi_0\otimes \Ku  + \Ku \otimes\Psi_0;\ \Phi_0, \Psi_0 \neq
c\Ku $
	\end{enumerate}
	
	In what follows, $c$ and $c_i$ represent complex 
	constants that may change from paragraph to paragraph.

	Case 1: We have $\varphi_1\otimes \psi_2
	=(\SC{\varphi_2})\otimes(\SC{\psi_1})$, which means that $\psi_2$
	and $\SC{\psi_1}$ are linearly dependent. 
	Both $\psi_2$ and $\SC{\psi_1}$ are non-zero, so $\psi_2 =
\overline{\alpha}
	\SC{\psi_1}$ for $\alpha\neq 0$ and it follows from Proposition~\ref{prop:inBalpha} that $A_\Psi$ is of type $\alpha$.

	Case 2: We have $\varphi_1\otimes \psi_2 - 
	(\SC{\varphi_2})\otimes(\SC{\psi_1}) = f \otimes g;\ f,g\in\modsp$.	
	So either $\varphi_1$ and $\SC{\varphi_2}$ are linearly dependent or 
	$\SC{\psi_1}$ and $\psi_2$ are. In the latter case, we again get that
	$A_\Psi$ is of type $\alpha$ for some $\alpha\neq 0$.
	Assume instead that $\varphi_1 = c_1\SC{\varphi_2}$ for $c_1\neq0$. Then by Equation~\ref{eqn:SCSC} $c_2\SC{\varphi_1} =\SC{\SC{\varphi_2}} = \varphi_2 - \ip{\varphi_2}{\Ku}\Ku$, and so $\varphi_2 - c_2\SC{\varphi_1} \in\spKu$ and therefore by Proposition~\ref{prop:inBalpha} $A_\Phi$ is of type $\alpha = \overline{c_2}$.	
	
	Case 3: We have
	$\varphi_1\otimes \psi_2- (\SC{\varphi_2})\otimes(\SC{\psi_1})
	= \Phi_0\otimes \Ku  + \Ku \otimes\Psi_0; \Phi_0, \Psi_0 \neq c\Ku$.
There exists $f\in\modsp$
	such that $f(0)=0$ and $\ip{f}{\Phi_0}=1$. Then we have 
	\begin{align*}
		\Ku &= \left(\Psi_0\otimes\Ku + \Ku\otimes\Phi_0\right)f\\
		&= \left(\psi_2\otimes\varphi_1\right)f -
			\left(\SC{\psi_1}\otimes\SC{\varphi_2}\right)f\\
		&= \psi_2\ip{f}{\varphi_1} - \SC{\psi_1}\ip{f}{\SC{\varphi_2}}
	\end{align*}

	If $\ip{f}{\varphi_1} = 0$, then $c\Ku = \SC{\psi_1}$, and so by Proposition~\ref{prop:inBalpha} $A_\Psi$ is of type $\infty$.
	Similarly, if $\ip{f}{\SC{\varphi_2}}=0$, then $c\Ku = \psi_2$ and $A_\Psi$ is of type $0$.
	So we can assume that $\psi_2 = \overline{\alpha}\SC{\psi_1} +
c\Ku$ for
	some $\alpha\neq 0$. Thus $A_\Psi$ is of type $\alpha$ by Proposition~\ref{prop:inBalpha}.
	\end{proof}

\begin{exmp}\label{exmp:FRTTO}
		Theorem 5.1 of ~\cite{sarason} classifies all the rank one operators in $\TTOs$ and finds symbols for them. Specifically, for $\lambda\in\disc$ $\CKu[\lambda]\otimes\Ku[\lambda]$ is in $\TTOs$ and has with symbol $u/(z-\lambda)$, and if $u$ has an ADC at $\zeta\in\torus$ then $\Ku[\zeta]\otimes\Ku[\zeta]$ is in $\TTOs$ and has symbol $\Ku[\zeta] + \overline\Ku[\zeta] - 1$. We will show that all of them are of type $\alpha$ for some $\alpha\in\cplane$, and compute $\alpha$.
	
		Let $\lambda\in\disc$ and consider $A = \CKu[\lambda]\otimes\Ku[\lambda]$, with symbol $u/(z-\lambda)$. Since $\CKu[\lambda](\lambda) = u'(\lambda)$, $\left(\CKu[\lambda]\otimes\Ku[\lambda]\right)^2 = u'(\lambda)\CKu[\lambda]\otimes\Ku[\lambda]$ so it follows that $\CKu[\lambda]\otimes\Ku[\lambda]$ is of type $\alpha$ for some $\alpha\in\plane^*$. Since 
\begin{align*}
u/(z-\lambda) &\TTOeq \CKu[\lambda] + u(\lambda)/(z-\lambda)\\&\TTOeq \CKu[\lambda] + u(\lambda)\overline{zK_\lambda} \\&\TTOeq \CKu[\lambda] + u(\lambda)\overline{\CS \Ku[\lambda]}
 \end{align*}
 $A$ is of type $u(\lambda)$.

		Now instead suppose that $\zeta\in\torus$ such that $u$ has an ADC at $\zeta$, and consider $A = \Ku[\zeta]\otimes\Ku[\zeta]$ which has symbol $\Ku[\zeta] + \overline{\Ku[\zeta]} - 1$. Again it is clear that $A^2$ is a scalar multiple of $A$ and hence $A$ is of type $\alpha$ for some $\alpha$. Since $A$ is self-adjoint, it follows that $\alpha$ is unimodular. We compute 
\begin{align*}
	\CKu[\zeta] &=\frac{u-u(\zeta)}{z-\zeta}\\
	&= \frac{u(\zeta)\left(1-\overline{u(\zeta)}u\right)}{\zeta\left(1-\overline{\zeta}z\right)}\\
	&= \overline{\zeta}u(\zeta)\Ku[\zeta]
\end{align*}
 so \[\SC{\Ku[\zeta]} = \zeta\CKu[\zeta] - u(\zeta)\Ku = u(\zeta)\left(\Ku[\zeta] - \Ku\right)\] Thus $\Ku[\zeta] - 1 \TTOeq \overline{u(\zeta)}\SC{\Ku[\zeta]}$ and so $\Ku[\zeta] + u(\zeta)\overline{\SC{\Ku[\zeta]}}$ is a symbol for $A$, which is therefore of type $u(\zeta)$.
	\end{exmp}
	
Theorem~\ref{thm:result1} has the following consequence which is an analogue of Corollary 2 in~\cite{brownhalmos}.

\begin{thm}\label{thm:TTOinverse}
  Let $A\in\TTOs$ be invertible. Then $A^{-1}\in\TTOs$ if and only if $A$ is of
type $\alpha$ for some $\alpha\in\plane^*$. If $A^{-1}\in\TTOs$, then $A$ and $A^{-1}$ are of the same type
\end{thm}
\begin{proof}
  If $A^{-1}\in\TTOs$, then both $A$ and $A^{-1}$ are of type $\alpha$ for some
$\alpha\in\plane^*$ by Theorem~\ref{thm:result1} since their product is $I = A_{\Ku}$. If $A$
is of type $\alpha$, either $|\alpha|\leq 1$ or $A^*$ is of type $\beta =
1/\overline{\alpha}\leq 1$. In the first case, we have that $A \CSalpha =
\CSalpha A$, so $A^{-1} \CSalpha = A^{-1} \CSalpha A A^{-1} = A^{-1} A \CSalpha
A^{-1} = \CSalpha A^{-1}$ and $A^{-1}$ is a TTO of type $\alpha$. In the second case, we have that $A^*$ is an
invertible TTO of type $\beta$ where $|\beta|\leq 1$, so its inverse is a TTO of
type $\beta$ as well. By taking adjoints again, the result follows.
\end{proof}

	 $\plane I$ is a subalgebra of $\Balpha$ for every $\alpha$, and the intersection of $\Balpha$ and $\Balpha[\beta]$ is either $\Balpha$ or $\plane I$ depending on whether $\alpha = \beta$ or not. We now consider an arbitrary algebra  $\mathcal{A}$ contained in $\TTOs$ and its relationship to $\Balpha$.
\begin{thm}
	Let $\mathcal{A}$ be an algebra contained in $\TTOs$. 
	Then there exists an $\alpha\in\cplane$ such that $\mathcal{A}$ is a subalgebra of $\Balpha$. 
\end{thm}

\begin{proof} Suppose every $A$ in $\mathcal{A}$ is of the form $cI$, for $c\in\plane$. Then $I\in\mathcal{A}$ and so $\mathcal{A} =\plane I$ which is a subalgebra of every $\Balpha$. 
	
	 Suppose then that there is $A\in\mathcal{A}$ not of the form $cI$. $A^2\in\TTOs$ so by Theorem~\ref{thm:result1} $A$ is of type $\alpha$ for some unique $\alpha$. If $B\in\mathcal{A}$ then $AB\in\TTOs$ and so since $A \neq cI$ it follows from Theorem~\ref{thm:result1} that $B$ is of type $\alpha$ as well, and therefore every operator in $\mathcal{A}$ is of type $\alpha$, and so it is a subalgebra of $\Balpha$
	\end{proof}

\section{Properties of $\Balpha$}
Due to the duality between $\Balpha$ and $\Balpha[\left(\overline{\alpha}^{-1}\right)]$ via taking adjoints, in order to study these algebras we can look at the cases where $\alpha\in\cdisc$. These algebras can then be divided into two different groups, $\alpha\in\disc$ and $\alpha\in\torus$. Different techniques are needed to deal with each of these cases. We discuss what the product of two TTOs of type $\alpha$ is, and expand on Theorem~\ref{thm:TTOinverse} by finding necessary and sufficient conditions for a TTO of type $\alpha$ to be invertible, based on its symbol.

\subsection{$\alpha\in\disc$} In this subsection, assume $\alpha\in\disc$.

Sarason's Commutant Lifting Theorem~\cite{sarasonold} states that if $A$ is a bounded operator that commutes with $\CS$, then there exists a function $\varphi\in H^\infty$ such that $\|A\|=\|\varphi\|_\infty$ and $A =A_\varphi$. The goal of this subsection is to find a Commutant Lifting Theorem for $\cmt{\CSalpha}$.

Let $u_\alpha = \frac{u-\alpha}{1-\overline{\alpha}u}$ for $\alpha\in\disc$. In what follows, we will be dealing with operators in both $\TTOs$ and $\TTOsa$. Let $A_\Phi^u$ refer to an operator in $\TTOs$ and $A_\Phi^{u_\alpha}$ an operator in $\TTOsa$.

$T_\alpha = M_{(1-|\alpha|^2)^{-1/2}(1-\overline{\alpha}u)}$ is an unitary map from $\modspa$ onto $\modsp$ called a Crofoot transform~\cite{crofoot}. Note that $T_\alpha^{-1} = M_{(1-|\alpha|^2)^{1/2}(1-\overline{\alpha}u)^{-1}}$. Sarason ~\cite{sarason} showed that $\CSalpha = A^u_{z/(1-\alpha\overline{u})}$ and that $T_\alpha^{-1}\CSalpha T_\alpha = A^{u_\alpha}_z$, the compressed shift on $\modspa$. Thus there is a unitary equivalence between $\Balpha$ on $\modsp$ and $\Balpha[0]$ on $\modspa$. The following propositions describe the operators of the form $A^u_{\varphi/(1-\alpha\overline{u})}$ for $\varphi\in H^2$, which are in fact the operators in $\Balpha$.
\begin{prop} \mbox{ }
	\begin{enumerate}
	\item		For $\varphi\in\modsp$ and $\alpha\in\disc$, $A^u_{\varphi/(1-\alpha\overline{u})} = A^u_{\varphi(1+\alpha\overline{u})} = A^u_{\varphi-\alpha\overline{\SC{\varphi}}}$.
	\item		If $\varphi\in H^2$, then $A^u_{\overline{\varphi}/(1-\alpha\overline{u})} = A^u_{\overline{\varphi}}$. Specifically, $A^u_{(1-\alpha\overline{u})^{-1}} = I$.
	\item		$\CSalpha = A^u_{z/(1-\alpha\overline{u})}$.
	\end{enumerate}
\end{prop} 
\begin{proof} \mbox{ }
	 \begin{enumerate}
	 \item		Since
   \[\frac{1}{1-\alpha\overline{u}} = \sum_{n=0}^\infty (\alpha\overline{u})^n\] we can compute 
   \[\frac{\varphi}{1-\alpha\overline{u}} = \sum_{n=0}^\infty \varphi(\alpha\overline{u})^n\] But since $\overline{u}\varphi\in\overline{zH^2}$ it follows that \(\sum_{n=0}^\infty \varphi(\alpha\overline{u})^n \TTOeq \varphi(1+\alpha\overline{u})\) and so $A^u_{\varphi/(1-\alpha\overline{u})}=A^u_{\varphi(1+\alpha\overline{u})}$. The second equality then hold because \[\overline{\SC{\varphi}} = \overline{\widetilde{\CS^* \varphi}} = \overline{u}z\frac{\varphi-\varphi(0)}{z} \TTOeq \varphi\overline{u}.\]
   \item	$\overline{\varphi}/(1-\alpha\overline{u})\TTOeq \overline{\varphi} + \alpha\overline{u\varphi}/(1-\alpha\overline{u})\TTOeq \overline{\varphi}$ by Proposition~\ref{prop:TTOzerocond}, since $\overline{u\varphi}/(1-\alpha\overline{u})\in\overline{uH^2}$.
   \item  Equation~(\ref{eqn:CSalphaSymbol}) and part $(1)$ of this proof imply that $\CSalpha$ has symbol \[\frac{1}{1-\alpha\overline{u(0)}}\left(\frac{\CS\Ku}{1-\alpha\overline{u}}+\alpha\overline{u'(0)}\right)\] so it suffices to show that \begin{equation*}\frac{z(1-\alpha\overline{u(0)})}{1-\alpha\overline{u}} \TTOeq \frac{\CS\Ku}{1-\alpha\overline{u}}+\alpha\overline{u'(0)}.\end{equation*} Since $z = \CS\Ku + uP(\overline{u}z)$,
	\begin{align*}\frac{z}{1-\alpha\overline{u}} &\TTOeq \frac{\CS\Ku}{1-\alpha\overline{u}} + \frac{uP(\overline{u}z)}{1-\alpha\overline{u}}\\&\TTOeq\frac{\CS\Ku}{1-\alpha\overline{u}} + \frac{\alpha P(\overline{u}z)}{1-\alpha\overline{u}}.\end{align*} Since $\CKu = \left(u-u(0)\right)\overline{z}$, $P(\overline{u}z) = \overline{\CKu(0)} + \overline{u(0)}z = \overline{u'(0)} + \overline{u(0)}z$, \begin{align*}\frac{z(1-\alpha\overline{u(0)})}{1-\alpha\overline{u}} &\TTOeq \frac{z}{1-\alpha\overline{u}} -\frac{\alpha\overline{u(0)}z}{1-\alpha\overline{u}}\\&\TTOeq \frac{\CS\Ku}{1-\alpha\overline{u}} + \frac{\alpha\overline{u'(0)}}{1-\alpha\overline{u}}+ \frac{\alpha\overline{u(0)}z}{1-\alpha\overline{u}}- \frac{\alpha\overline{u(0)}z}{1-\alpha\overline{u}}\\&\TTOeq \frac{\CS\Ku}{1-\alpha\overline{u}} + {\alpha\overline{u'(0)}}.\end{align*}
   \end{enumerate}
\end{proof}

\begin{lem}\label{lem:uandualpha}
	Let $\varphi\in H^2$ and $\alpha\in\disc$. Then $T_\alpha A^{u_\alpha}_\varphi T_\alpha^{-1} = A^{u}_{\varphi/(1-\alpha\overline{u})}$ and $T_\alpha A^{u_\alpha}_{\overline{\varphi}} T_\alpha^{-1} = A^{u}_{\overline{\varphi}/(1-\overline{\alpha}u)}$. Therefore $A^{u_\alpha}_\varphi$ and  $A^{u}_{\varphi/(1-\alpha\overline{u})}$ (respectively $A^{u_\alpha}_{\overline{\varphi}}$ and  $A^{u}_{\overline{\varphi}/(1-\overline{\alpha}u)}$)  have the same norm, and if $\psi\in H^2$, then  $A^{u}_{\varphi/(1-\alpha\overline{u})}=A^u_{\psi/(1-\alpha\overline{u})}$ (respectively $A^{u}_{\overline{\varphi}/(1-\overline{\alpha}u)}=A^u_{\overline{\psi}/(1-\overline{\alpha}u)}$) if and only if $u_\alpha | (\varphi - \psi)$.
\end{lem}
\begin{proof}
It suffices to show that the equalities hold on $K_u^\infty$, so let $f\in K_u^\infty$. Then \[A^{u}_{\varphi/(1-\alpha\overline{u})}f = P_u\left(\frac{f\varphi}{1-\alpha\overline{u}}\right) = P\left(\frac{f\varphi}{1-\alpha\overline{u}}\right) - u P \left(\frac{\overline{u}f\varphi}{1-\alpha\overline{u}}\right)\] On the other hand, 
\begin{align*}
	T_\alpha A^{u_\alpha}_\varphi T_\alpha^{-1}f &= \left(1-\overline{\alpha}u\right)P_{u_\alpha}\left(\frac{f\varphi}{1-\overline{\alpha}u}\right)\\
	&= \left(1-\overline{\alpha}u\right)\left[\frac{f\varphi}{1-\overline{\alpha}u}-u_\alpha P\left( \frac{\overline{u_\alpha}f\varphi}{1-\overline{\alpha}u}\right)\right]\\
	&= f\varphi-(u-\alpha) P \left(\frac{\overline{u}f\varphi}{1-\alpha\overline{u}}\right)\\
	&= f\varphi + P\left(\frac{\alpha\overline{u}f\varphi}{1-\alpha\overline{u}}\right) - u P \left(\frac{\overline{u}f\varphi}{1-\alpha\overline{u}}\right)\\
	&=P\left(\frac{f\varphi}{1-\alpha\overline{u}}\right) - u P \left(\frac{\overline{u}f\varphi}{1-\alpha\overline{u}}\right)
\end{align*}

Since $T_\alpha$ is unitary, it follows that $A^u_{\varphi/(1-\alpha\overline{u})}= A^u_{\psi/(1-\alpha\overline{u})}$ if and only if $A^{u_\alpha}_\varphi = A^{u_\alpha}_\psi$, but by Proposition~\ref{prop:TTOzerocond} the latter is true if and only if $u_\alpha | \varphi -\psi$.

Since $T_\alpha$ is unitary, we have
\begin{align*}
	A^u_{\overline{\varphi}/(1-\overline{\alpha}u)} &= \left (A^u_{\varphi/(1-\alpha\overline{u})}\right)^* \\
	&= \left(T_\alpha A^{u_\alpha}_\varphi T_\alpha^{-1}\right)^* \\
&= T_\alpha A^{u_\alpha}_{\overline{\varphi}} T_\alpha^{-1}
\end{align*}
proving the result for the adjoints.

\end{proof}

\begin{thm}
	Let $A$ be an bounded operator on $\modsp$ and let $\alpha\in\disc$. Then $A$ is of type $\alpha$ if and only if there is a function $\varphi\in H^2$ such that $A = A^{u}_{\varphi/(1-\alpha\overline{u})}$. If $A$ is of type $\alpha$ then there is a function $\psi\in H^\infty$ such that $\|\psi\|_\infty = \|A\|$ and $A = A^{u}_{\psi/(1-\alpha\overline{u})}$ and therefore every operator of type $\alpha$ has a bounded symbol. Further, if $\varphi, \psi$ are in $H^\infty$ then $A^u_{\varphi/(1-\alpha\overline{u})} A^u_{\psi/(1-\alpha\overline{u})} = A^u_{\varphi\psi/(1-\alpha\overline{u})}$.
\end{thm}

\begin{proof}
	Let $B = T_\alpha^{-1}AT_\alpha$. Then \[AA^{u}_{z/(1-\alpha\overline{u})} = A^{u}_{z/(1-\alpha\overline{u})}A\] if and only if \[BA^{u_\alpha}_z = T_\alpha^{-1}AA^{u}_{z/(1-\alpha\overline{u})}T_\alpha = T_\alpha^{-1}A^{u}_{z/(1-\alpha\overline{u})}AT_\alpha = A^{u_\alpha}_zB\] But this is true if and only if $B = A^{u_\alpha}_\varphi$ for some $\varphi\in H^2$ which is true if and only if $A = A^{u}_{\varphi/(1-\alpha\overline{u})}$ for some $\varphi\in H^2$, hence the first claim holds. By the Commutant Lifting Theorem, there is a function $\psi\in H^\infty$ such that $A^{u_\alpha}_\varphi = A^{u_\alpha}_\psi$ and $\|A^{u_\alpha}_\varphi\| = \|\psi\|_\infty$. By Lemma~\ref{lem:uandualpha} it follows that $A = A^u_{\psi/(1-\alpha\overline{u})}$. Since $T_\alpha$ is unitary, $\|A\| = \|\psi\|_\infty$.
	
To prove the last claim, we compute \[A^u_{\varphi/(1-\alpha\overline{u})} A^u_{\psi/(1-\alpha\overline{u})} = T_\alpha^{-1}  A^{u_\alpha}_{\varphi} A^{u_\alpha}_{\psi} T_\alpha = T_\alpha^{-1} A^{u_\alpha}_{\varphi\psi} T_\alpha =   A^u_{\varphi\psi/(1-\alpha\overline{u})}\]
\end{proof}

Just as $A^u_\varphi = \varphi(\CS)$ for $\varphi\in H^\infty$, we get that $A^u_{\varphi/(1-\alpha\overline{u})} = \varphi\left(\CSalpha\right)$ for $\varphi\in H^\infty$.

Note that $\lambda$ is in the spectrum of $A^u_\varphi$ if and only if $\inf_{z\in\disc}(|u(z)| + |\varphi(z)-\lambda|) = 0$~\cite{CMR}. 
\begin{prop}
Let $\alpha\in\disc$ and let $\varphi\in H^\infty$. Then $A^u_{\varphi/(1-\alpha\overline{u})}$ is invertible if and only if $\inf_{z\in\disc}(|u_\alpha(z)| + |\varphi(z)|) > 0$
\end{prop}
\begin{proof}

$A^u_{\varphi/(1-\alpha\overline{u})}$ is invertible if and only if $A^{u_\alpha}_\varphi$ is invertible, which is true if and only if  $\inf_{z\in\disc}(|u_\alpha(z)| + |\varphi(z)|) > 0$.
\end{proof}
\subsection{$\alpha\in\torus$}

The case of $|\alpha| = 1$ is indirectly dealt with in~\cite{sarason,arXiv} and we collect those results here.
There are TTOs of unimodular type without a bounded symbol under certain conditions. Specifically, in \cite{arXiv} it is shown that there exists $u$ an inner function with an ADC at $\zeta\in\torus$ such that $\Ku[\zeta]\otimes\Ku[\zeta]\in\TTOs$ does not have a bounded symbol.

 Example~\ref{exmp:FRTTO} shows that $\Ku[\zeta]\otimes\Ku[\zeta]$ is of type $u(\zeta)$, and hence it is an example of a TTO of unimodular type without a bounded symbol.

If, however, we weaken what we mean by ``bounded symbol'' we can find a bounded symbol for any TTO of unimodular type. Specifically, we change the measure with respect to which we take the sup norm of a function. 

Let $\alpha$ be unimodular, and fixed for the rest of this section. An operator is of type $\alpha$ if and only if it commutes with $\CSalpha$, which is in this case a unitary operator known as a Clark unitary operator, and is unitarily equivalent to $M_z$ on the space $L^2(\torus, \mu_\alpha)$ where $\mu_\alpha$ is the Clark measure associated with $\CSalpha$~\cite{clark}. $\cmt{M_z}$ is the space of multiplication operators induced by $L^\infty(\mu_\alpha)$ and so by using the unitary equivalence, every operator of type $\alpha$ is equal to $\Phi(\CSalpha)$ where $\Phi\in L^\infty(\mu_\alpha)$. In this sense we can think about $\Phi$ as a ``bounded symbol'' for the operator. This gives us a symbol calculus of sorts for operators of type $\alpha$: given $\Phi,\Psi$ bounded $\mu_\alpha$-almost everywhere, the product of $M_\Phi$ and $M_\Psi$ is $M_{\Phi\Psi}$ where $\Phi\Psi$ is itself bounded $\mu_\alpha$-almost everywhere. Hence $\Phi(\CSalpha)\Psi(\CSalpha) = \Phi\Psi(\CSalpha)$. It follows that a TTO of type $\alpha$ is invertible if and only if it is of the form $\Phi(\CSalpha)$, where $|\Phi|\geq \delta > 0$ $\mu_\alpha$-almost everywhere.

We can use this symbol calculus to precisely describe the unitary operators in $\TTOs$ on a given model space.

\begin{prop}\label{prop:TTOunitary}
	Let $A\in\TTOs$. Then $A$ is unitary if and only if it is equal to $\Phi(\CSalpha)$ for some $\alpha\in\torus$ and some $\Phi\in L^\infty(\torus,\mu_\alpha)$ such that $|\Phi| = 1$ $\mu_\alpha$-almost everywhere. Specifically, any unitary operator in $\TTOs$ is of unimodular type, and commutes with the Clark unitary operator of the same type.
\end{prop}
\begin{proof}
	If $A$ is unitary then $A A^* = I$, which means that $A$ and $A^*$ must both be of the same type $\alpha\in\cplane$. Thus $\alpha = \overline{\alpha}^{-1}$ which implies that $\alpha$ is of unimodular type. So $A = \Phi(\CSalpha)$ for some $\Phi\in L^\infty(\torus,\mu_\alpha)$. Then $I = A A^* = \Phi(\CSalpha)\overline{\Phi}(\CSalpha) = |\Phi|^2(\CSalpha)$ which implies that $|\Phi| = 1$ $\mu_\alpha\mathrm{-almost}$ everywhere. The other direction is obvious.
\end{proof}

\providecommand{\bysame}{\leavevmode\hbox to3em{\hrulefill}\thinspace}
\providecommand{\MR}{\relax\ifhmode\unskip\space\fi MR }
\providecommand{\MRhref}[2]{%
  \href{http://www.ams.org/mathscinet-getitem?mr=#1}{#2}
}
\providecommand{\href}[2]{#2}

\end{document}